\documentclass[reqno,11pt]{amsart}
\usepackage{lineno,amsmath,amsthm,amssymb,graphics,epsfig,mathrsfs}
\usepackage{enumerate,geometry,etoolbox, scalerel, stackengine, relsize}
\usepackage{amscd,latexsym,url,upgreek,mathtools,color}
\usepackage{mathrsfs}
\usepackage[english]{babel}

\newcommand{\newabstract}[1]{%
  \par\bigskip
  \csname otherlanguage*\endcsname{#1}%
  \csname captions#1\endcsname
  \item[\hskip\labelsep\scshape\abstractname.]
}

\newtheorem{theorem}{Theorem}[section]
\newtheorem{lemma}[theorem]{Lemma}
\newtheorem{corollary}[theorem]{Corollary}
\newtheorem{proposition}[theorem]{Proposition}
\newtheorem{example}[theorem]{Example}

\theoremstyle{definition}
\newtheorem{remark}[theorem]{{\bf Remark}}
\newtheorem{definition}[theorem]{Definition}

\def\C{\mathbb C}

\def\R{\mathbb R}

\def\N{\mathbb N}

\def\Z{\mathbb Z}

\def\D{\mathbb D}

\def\bfK{{\boldsymbol K}}
\def\bfU{{\boldsymbol U}}

\def\bfk{{\boldsymbol k}}

\def\({{\rm (}}
\def\){{\rm )}}

\begin{document}

\numberwithin{equation}{section}

\title{Analyticity and supershift with regular sampling}

\author[F. Colombo]{F. Colombo}
\address{(FC) Politecnico di
Milano\\Dipartimento di Matematica\\Via E. Bonardi, 9\\20133 Milano\\Italy}
\email{fabrizio.colombo@polimi.it}

\author[I. Sabadini]{I. Sabadini}
\address{(IS) Politecnico di
Milano\\Dipartimento di Matematica\\Via E. Bonardi, 9\\20133 Milano\\Italy}
\email{irene.sabadini@polimi.it}

\author[D. C. Struppa]{D. C. Struppa}
\address{(DCS) The Donald Bren Presidential Chair in Mathematics\\ Chapman University, Orange, CA 92866 \\ USA}
\email{struppa@chapman.edu}

\author[A. Yger]{A. Yger}
\address{(AY) IMB, Universit\'e de Bordeaux, 33405, Talence, France}
\email{yger@math.u-bordeaux.fr}

\thanks{The first and second author are supported by MUR grant Dipartimento di Eccellenza 2023-2027.}

\begin{abstract}
The notion of supershift (in itself a generalization of the notion of superoscillation arising in quantum mechanics) expresses the fact that the sampling of a function in an interval allows to compute the values of the function far from the interval.
In this paper we study the relation between supershift and real analyticity.
 We use a classical result due to Serge Bernstein to
show that real analyticity for a complex valued function implies a strong form of supershift. On the other hand, we use a
parametric version of a result by
Leonid Kantorovitch to show that the converse is not true.
We additionally study the stability of functions that satisfy suitable
supershift requirements under multiplication by complex numbers
 or primitivization.
\end{abstract}

\maketitle



\section{Introduction}
In their seminal paper \cite{AAV88}, Aharonov and his collaborators
introduced a new concept in quantum mechanics, the so-called {\it weak
value}, and as a consequence they obtained some apparently paradoxical
results (for a complete review of the notion of weak value see \cite{dressel}). One of the
reasons why such results were possible is the discovery of a family of
functions from ${\mathbb R}$ to ${\mathbb C}$ with the surprising
property that they are completely determined, on the real axis, by their
values on a countable subset of a finite interval in $\mathbb R$. The
first example of such a function is the exponential function $f(x,a) : a
\longrightarrow e^{iax}$ regarded as a function of the real variable
$a$, with $x$ as a parameter. Specifically, it is possible to write the
function $a\longmapsto e^{iax}$, for arbitrary large values of $a$, in
terms of the values $e^{i\lambda_{N,\nu}x}$ with
$\lambda_{N,\nu}=1-2\nu/N$ for $N\in \N^*$ and $0\leq \nu \leq N$, and
therefore $|\lambda_{N,\nu}|\leq 1.$ To show that this is the case, one
considers the function
\begin{equation}\label {Fn}
F_N(x,a):=\big(\cos\frac{x}{N}+ia\sin\frac{x}{N}\big)^N,
\end{equation}
which (by simply using the Euler formulas for sine and cosine, and the
Newton binomial formula) can be rewritten as
$$
F_N(x,a)=\sum_{\nu =0}^N C_\nu(N,a)e^{i(1-2\nu/N)x},
$$
where
\begin{equation}\label{exampleaharonov}
C_\nu(N,a):= {\binom{N}{\nu}\Big(\frac{1+a}{2}\Big)^{N-\nu}
\Big(\frac{1-a}{2}\Big)^{\nu}}.
\end{equation}
Since (\ref{Fn}) immediately implies that $\lim_{N\rightarrow +\infty}
\, F_N(x,a)=e^{iax}$, we obtain that the value of the function $a \longmapsto f(x,a) =e^{iax}$ for $a$ arbitrarily large
is retrievable if one knows the values of the function {on} the
{countable} set of values $\lambda_{N,\nu}\in [-1,1]$.

Since in this particular example the variable $a$ is the frequency of
the oscillating exponential function, now to be considered as a function
of $x$, this phenomenon is usually referred to as {\it superoscillation}
and we call a sequence such as $\{x \longmapsto F_N(x,a)\}_{N\geq 1}$ a
{\it superoscillating sequence}, and the limit function a {\it
superoscillating function.}

The fact that large oscillations can be obtained by superposition of
much smaller oscillations, so that a highly oscillating function could
be seen as a limit of band-limited functions, led to a series of
surprising physical phenomena, some of which are described in
\cite{APR91,APR20} and were recently the subject of a more popular
discussion in \cite{McCo22}.

One of the first questions that Aharonov posed was whether this interesting phenomenon would be conserved if one evolves a superoscillating sequence according to the Schr\"odinger equation for the free particle. As we have shown in \cite{ACSST13,ACSST15}, the answer is positive, indeed one can prove that
the solution to the Cauchy problem
$$
i\frac{\partial \psi(x,t)}{\partial t}=-\frac{\partial^2 \psi(x,t)}{\partial x^2},\ \ \ \ \psi(x,0)=F_N(x,a),
$$
is given by
\begin{equation}\label{psiNN}
\psi_N(x,t)=\sum_{k=0}^N C_k(N,a) e^{i(1-2k/N) x } e^{-i(1-2k/N)^2t},
\end{equation}
and that  $\psi_N(x,t)$ converges to $e^{iax-ia^2t}$ uniformly on every compact set in $\R^2$.

This result is of independent interest, but it also shows the existence of a simple case of superoscillations in two variables. This realization led to the development of an independent theory of superoscillations in several variables, see \cite{ACJSSST22,ACSST16}.
Once the case of the free particle is settled, the next natural question is to evolve a superoscillatory sequence with different potentials. The authors have obtained a large number of results in this direction but even the next simplest case, namely the harmonic oscillator studied in \cite{BCSS14}, highlights a completely new phenomenon. Indeed we proved that the Cauchy problem for the quantum harmonic oscillator
$$
i\frac{\partial \psi(t,x)}{\partial t}=\frac{1}{2}\Big( -\frac{\partial^2 }{\partial x^2} +x^2 \Big)\psi(t,x),\ \ \ \ \psi(0,x)=F_N(x,a)
$$
has the solution
\begin{equation}\label{psiNNN}
\psi_N(t,x)=(\cos t)^{-1/2}e^{-(i/2) x^2\tan t}\sum_{k=0}^NC_k(N,a) e^{ ix(1-2k/N)/\cos t -(i/2)(1-2k/N)^2\tan t}
\end{equation}
and when we take the limit for $N\to \infty$ we get
$$
\lim_{N\to \infty}\psi_N(t,x)=(\cos t)^{-1/2} e^{ -(i/2) x^2\tan t} e^{ -(i/2) a^2\tan t +iax/\cos t}.
$$

As one can see, the function in \eqref{psiNNN} is not superoscillating as the function in \eqref{psiNN}. The phenomenon that arises is now the fact that a function can be described in points away from the origin if one knows its value in enough points near the origin. This phenomenon (of which superoscillations is the simplest example) is what we called supershift (a precise definition was given in \cite{CGS19} and will be repeated later on in Section 2). As it turned out, we were able to demonstrate this phenomenon in a very large number of cases, and we noticed that the possibility of doing so seemed to be a direct consequence of the analyticity of the functions for which supershift was to be demonstrated. This raises the very natural question of whether the two notions are indeed connected. In this paper we will show that in fact analyticity implies supershift (this result is already implied in some of our previous works) but that the converse is not true. This result goes full circle because one of Aharonov's early remarks was that, in some naive sense,
the property of superoscillations can be considered akin to analyticity.
Indeed, an analytic function is completely determined by {countably}
many data at the origin (the value, and all of its derivatives) and a
superoscillating function is completely determined by {countably} many
data near the origin (its values in the points $\lambda_{N,\nu}$).

The plan of the paper is the following.
Section 2 contains the various definitions and refinements of the notions of superoscillation and supershift that we shall use later.
In Section \ref{sect2new} we show that the notion of analyticity is
stronger than a suitably formalized notion of supershift.
Specifically, in Theorem \ref{sect2new-prop2}, using a parametric reformulation of a classical result due to Serge Bernstein we
show that real analyticity for  a complex valued function implies a strong form of supershift. On the other hand, we use a
parametric version of a result by
Leonid Kantorovitch to show, see Theorem
\ref{sect2new-prop3}, that the converse is not true.
In Section 4 we additionally study the stability of functions that satisfy suitable
supershift requirements under essential operations such as
complex multiplication or primitivization (Theorems
\ref{sect2new-thm3} and \ref{sect2new-thm4}).

\section{Preliminary results}
In this section we will offer some preliminary definitions and results that will be useful in the sequel. To begin with,
let
\begin{equation}\label{sect1-eq0}
H = \begin{bmatrix}
& h_{1,0},...,h_{1, \nu(1)}\\
& h_{2,0},h_{2,1},..., h_{2, \nu(2)}\\
& \vdots \\
& h_{N,0},h_{N,1},h_{N,2},...,h_{N,\nu(N)}\\
& \vdots  \end{bmatrix}
\end{equation}
be a collection of real numbers (interpreted as frequencies) which all belong to the closed interval $[-1,1]$ {with $h_{N,\nu} \geq h_{N,\nu+1}$ for $N\in \N^*$ and $0\leq \nu \leq \nu(N)-1$}.
The simplest example of such double index sequence can be obtained by setting $\nu(N)=N$ and $h_{N,\nu}=1-2(\nu+1)/N$, $0\leq \nu \leq N-1$.
\begin{definition}
A sequence of generalized trigonometric polynomials $(T_{H,N})_{N \geq 1}$  of the form
$$
{T_{H,N}}(x) = \sum\limits_{\nu=0}^{\nu(N)} C_{H,\nu}(N)\, \exp( ix h_{N,\nu}),\quad ({\rm with}\ C_{H,\nu}(N)
\in \C \quad {\rm for}\ N \in \N^*, \ 0\leq \nu\leq \nu(N))
$$
is called an {\it $H$-sequence of generalized trigonometric polynomials}.
\end{definition}
Given an open subset $A\subset \R$, an {\it $(A,H)$-sequence of generalized trigonometric polynomials}
$\{T_{H,N}[a]:a\in A\}_{N\geq 1}$ is by definition an $A$-parametrized $H$-sequence of generalized trigonometric polynomials such that each complex amplitude $C_{H,\nu}(N,a)$
depends continuously on
$a\in A$, for any $N\in \N^*$ .

\begin{definition}[Superoscillating sequence] An $(A,H)$-sequence of generalized trigonometric polynomials is said to be {\it superoscillating} if there are two continuous functions $g_H~: A \rightarrow \C$ and $C_H : A \rightarrow \R$ such that both $V=\{a \in A\,:\, C_H(a) \not=0,\ |g_H(a)| > 1\}$ and the open subset $U$ of points $x\in \R$ about which
\begin{equation}\label{sect1-eq1ter}
\lim_{N\rightarrow +\infty} T_{H,N}[a]\, (x) = C_H(a)\, \exp (i\, g_H(a)\, x),
\end{equation}
locally uniformly with respect to $(a,x)$, are non-empty open subsets respectively of $A$ and
$\R$.  The set $U\cap V$ is then called {\it the superoscillating subset} of the $(A,H)$-superoscillating sequence $\{T_N[a]\,:\, a\in A\}_{N\geq 1}$.
\end{definition}
\begin{remark}
What motivates the terminology here is the fact that, for any $a\in A$, each
$T_{H,N}[a]$ is a low-band signal, with frequencies in $[-1,1]$, while the signal
$x \mapsto C_H(a_0)\, e^{i g_H(a_0) x}$ is a true elementary oscillating signal with frequency $g_H(a_0)$ outside
$[-1,1]$ as soon as
$$
a_0 \in \{a\in A\,:\, C_H(a)\not=0,\ |g_H(a)| >1\} \not= \emptyset.
$$
\end{remark}

For a further insight to superoscillations see also \cite{bcss23}.
Throughout this paper, we shall consider $(\R,H)$-superoscillating sequences subordinate to the choice of $g= {\rm Id}_\R$ and $C\equiv 1$.

Generalizing the original description of Aharonov, it is easy to realize prototypical examples of superoscillating $(\mathbb{R},H)$-sequences of generalized trigonometric polynomials as we describe below.

\begin{definition}[Regular sampling]
We will call regular sampling of the frequency interval $[-1,1]$ the sampling defined by $H^{\boldsymbol \epsilon}=[h_{N,\nu}^{\boldsymbol \epsilon}]$ where
$\boldsymbol \epsilon= (\epsilon_N)_{N\geq 1}$ is a sequence of elements in $[0,1[$ which tends to $0$ when $N$ tends to infinity and
\begin{equation}\label{sect1-eq3}
h_{N,\nu}^{\boldsymbol \epsilon} = 1 - 2\, \Big(\frac{\nu + \epsilon_N (N-\nu)}{N}\Big), \ 0\leq \nu \leq N~.
\end{equation}
\end{definition}
\begin{remark}
The $(\R,H^{\boldsymbol \epsilon})$-sequence $\{T_N^{\boldsymbol \epsilon} [a]\,:\, a\in \R\}_{N\geq 1}$, where
\begin{equation}\label{sect1-eq4}
\begin{split}
T_{N}^{\boldsymbol \epsilon}[a]\, (x) & = \sum\limits_{\nu=0}^N \binom{N}{\nu}\, \Big(\frac{1+a}{2}\Big)^{N-\nu}
\Big(\frac{1-a}{2}\Big)^{\nu}
\exp \big( i h_{N,\nu}^{\boldsymbol \epsilon}\, x\big) \\
& = \exp (i\, \epsilon_N x)\, \Big(
\frac{1+a}{2}\, \exp \Big(ix\, \frac{1-\epsilon_N }{N} \Big) +
\frac{1-a}{2}\, \exp \Big(- ix\, \frac{1-\epsilon_N}{N}\Big)\Big)^N \\
& = \exp (i\, \epsilon_N x)\, \Big( \cos
\Big(x\, \frac{1-\epsilon_N}{N}\Big) + i\, a\, \sin \Big(x\, \frac{1-\epsilon_N}{N}\Big)\Big)^N
\end{split}
\end{equation}
is superoscillating on $U=\R$.
\end{remark}

\begin{remark}
When the array $H$ in \eqref{sect1-eq0} is such that
$\forall\, N \in \N^*,\ \nu(N)=N$ and so ${\rm card}\, \{h_{N,\nu}\,:\, \nu=0,...,\nu(N)\} = N+1$,
the sequence of Lagrange-Hermite interpolators (with respect to the parameter $a$)
\begin{equation}\label{sect1-eq5bis}
\Big\{ x \longmapsto T_{H,N}^{\rm Lag}[a](x) := \sum\limits_{\nu=0}^N
\Big(\prod_{\nu'\not=\nu}
\frac{a - h_{N,\nu'}}{h_{N,\nu} - h_{N,\nu'}}\Big)\, e^{ih_{N,\nu} x}\,:\, a\in \R\Big\}
\end{equation}
is also superoscillating since
\begin{equation}\label{sect1-eq5ter}
\begin{split}
\Big| e^{iax} -  T_{H,N}^{\rm Lag}[a](x)\Big| = \frac{\prod_{\nu=0}^N |a-h_{N,\nu}|}{(N+1)!}
\sup_{|\xi|\leq \max (1,|a|)} \Big|
\Big(\frac{\partial}{\partial a}\Big)^{N+1} [e^{iax}] (\xi)\Big| \leq
\frac{\big((|a|+1)|x|\big)^{N+1}}{(N+1)!}
\end{split}
\end{equation}
for any $(a,x)\in \R^2$ according to the expression of the reminder term in the Lagrange interpolation formula, see \cite[Theorem 2.2]{ACSSST21}. Such a superoscillating sequence $\{T_{H,N}^{\rm Lag}[a]\,:\, a\in \R\}_{N\geq 1}$ could be interpreted as corresponding to an irregular sampling of the low-frequency domain $[-1,1]$, which occurs as soon as an infinitely many rows $H_{N_\iota}$ of the array
$H$ are such that the function $\nu \in \{0,..., N_\iota -1\} \mapsto h_{N_\iota,\nu}- h_{N_\iota,\nu+1}>0$ is not constant.
\end{remark}

Let us introduce here the important concept of {\it Supershift Property $({\rm SP})_{\mathscr F}$}.

\begin{definition}\label{sect1-def1}(Supershift Property $({\rm SP})_{\mathscr F}$)
Let $A$ be an open interval of $\R$, possibly $\R$ itself, which contains $[-1,1]$ and let $\psi : a \in A \longmapsto \psi_a\in \mathscr F$ be a continuous map from
$A$ to a topological $\C$-vector space $\mathscr F$. Let
$$\left\{{T_{H,N}}[\lambda](x) = \sum\limits_{\nu=0}^{\nu(N)} C_{H,\nu}(N,\lambda)\, \exp( ix h_{N,\nu})\,:\, \lambda\in \R\right\}_{N\geq 1}
$$
be an
$(\R,H)$-superoscillating sequence. The map
$\psi$ is said to satisfy the {\it Supershift Property} $({\rm SP})_{\mathscr F}$ on $A$ with respect to $\{T_{H,N}[\lambda]\,:\, \lambda\in \R\}_{N\geq 1}$ if the sequence of functions
\begin{equation}\label{sect1-eq6}
a \in A \longmapsto \sum\limits_{\nu=0}^{\nu(N)}
C_{H,\nu} (N,a)\, \psi_{h_{N,\nu}} \in \mathscr F,\quad N=1,2,...
\end{equation}
converges to $\psi_a$ in the space of continuous functions $\mathcal C(A,\mathscr F)$, with respect to the topology of uniform convergence on any compact subset, i.e.
$$
\lim_{N\to\infty}\sum\limits_{\nu=0}^{\nu(N)}
C_{H,\nu} (N,a)\, \psi_{h_{N,\nu}}=\psi_a .
$$
\end{definition}
\begin{remark}
The reader will notice that this definition reduces,
 for {$C_{\nu} (N,a)$} as in
\eqref{exampleaharonov}, to the original example of Aharonov, which can in fact be obtained by taking $A={\mathbb{R}}$, and {$H=H^{\boldsymbol 0}:= \{1-2\nu/N\,:\, N\in \N^*,\ 0\leq \nu \leq N\}$}.
\end{remark}

Given a topological $\C$-vector space $\mathscr F$, we also introduce
for continuous maps from an open interval $A\subset \R$ with length strictly larger than $2 (=
{\rm length}\, ([-1,1])$) to $\mathscr F$, the so-called
{\it Translation-Commuting Supershift Property}
$({\rm TCSP})_{\mathscr F}$ with respect to an $(\R,H)$-superoscillating sequence.

\begin{definition}\label{sect1-def2} (Translation-Commuting Supershift Property $({\rm TCSP})_{\mathscr F}$)
Let $A$ be as above and for $a \in A$, let $\psi_a\in \mathscr F$ be a continuous map from
$A$ to $\mathscr F$. Let $\{T_{H,N}[\lambda]\,:\, \lambda\in \R\}_{N\geq 1}$ be a
$(\R,H)$-superoscillating sequence. The continuous map
$\psi$ is said to satisfy the {\it $\mathscr F$-Translation-Commuting Supershift Property $({\rm TCSP})_{\mathscr F}$} on $A$ with respect to the superoscillating sequence $\{T_{H,N}[\lambda]\,:\, \lambda\in \R\}_{N\geq 1}$ if the sequence of functions
\begin{equation}\label{sect1-eq8}
\left(\sum\limits_{\nu=0}^{\nu(N)}
C_{H,\nu} (N,a)\, \psi_{a' + h_{N,\nu}}\right) \subset \mathscr F,\quad N=1,2,...
\end{equation}
defined for
$$
(a,a') \in \mathbb A := \big\{(a,a')
\in \R \times A\,:\, a'+ [-1,1] \subset A,\ a+a'\in A\big\}
$$
converges to $\psi_{a + a'}$ in $\mathcal C(\mathbb A,\mathscr F)$ with respect to the topology of uniform convergence on any compact subset of $\mathbb A$.
\end{definition}
Observe that $({\rm TCSP})_{\mathscr F}$ on an open interval $A$ of $\mathbb R$ with respect to
such an $(\R,H)$ superoscillating sequence
implies $({\rm SP})_{\mathscr F}$ with respect to the same superoscillating sequence provided that $A$ contains $[-1,1]$.
\begin{remark}
If $A=[-\alpha, \beta]$ with $\alpha, \beta>1$, then the set $\mathbb A$ can be equivalently described as
$$
\mathbb A=\{(x,y)\in\mathbb R^2\ : 1 -\alpha\leq y\leq -1+\beta,\ -\alpha\leq x+y\leq \beta\}.
$$
\end{remark}
Let $\mathcal{H}(\mathbb{C})$ be the space of entire functions and let us consider
\begin{equation}\label{sect1-eq9}
{\rm Exp}(\C)=\{\Phi\in \mathcal{H}(\C)\ :\
\exists\, R >0 \quad {\rm such\ that}\quad \sup_{z\in \C} e^{-R |z|} |\Phi(z)| = 0\}.
\end{equation}
For any $R>0$, let also consider the subspace of  ${\rm Exp}(\C)$  defined by:
\begin{equation}\label{sect1-eq10}
{\rm Exp}_{<R}(\C)=\{\Phi\in \mathcal{H}(\C)\ :\
\exists\, \varepsilon \in\,  ]0,R]\quad {\rm such\ that}\quad
\sup_{z\in \C} e^{-(R-\epsilon)|z|} |\Phi(z)| < + \infty \}.
\end{equation}
Both spaces are topological spaces with their AU-structures, see
for example \cite[Definition 4.1.10]{ACSST17A}. We recall that the dual space of ${\rm Exp}_{< R}(\C)$
is isomorphic to $\mathcal{H}(\D_\C (0,R))$, see \cite[\textsection 6, Example 8]{Tayl68} or also \cite[Proposition 4.1.16]{ACSST17A}.\\

Let $A$ be an open interval of $\R$ which contains $[-1,1]$. Among the $(\R,H)$-superoscillating sequences,
those for which the sequence
\begin{equation}\label{sect1-eq11}
\Big( z \in \C  \longmapsto \sum\limits_{\nu=0}^{\nu(N)} C_{N,\nu}(N,a) e^{ih_{N,\nu} z}\Big)_{N\geq 1}
\end{equation}
converges locally uniformly with respect to $a\in A$ towards $z\mapsto e^{iaz}$
in ${\rm Exp}(\C)$
(or at least in ${\rm Exp}_{<R}(\C)$ for some $R> \sup_A |a|>1$ when $A$ is bounded)
will be of particular interest for us.

Let $\mathscr F$ be a topological $\C$-vector space and let $\bfK : {\rm Exp}(\C) \rightarrow \mathscr F$ be a continuous operator.
Then the continuous map
\begin{equation}\label{sect1-eq11bis}
\psi_{\bfK}~: a \in A \longmapsto \bfK (z\mapsto e^{iaz}),
\end{equation}
(or, more briefly, $\bfK (e^{iaz})$)
satisfies $({\rm SP})_{\mathscr F}$ on $A$ with respect to the $(\R,H)$-superoscillating sequence
which is considered here. The same argument applies to the continuous operator
$\bfK_{R} : {\rm Exp}_{< R} (\C) \rightarrow \mathscr F$
and the continuous map
$$
\psi_{\bfK_R}~: a\in A \longmapsto \bfK_{R} (e^{iaz})\ {\rm when}\ A\subset ]-R,R[.
$$
  Moreover, since the Taylor sequence
\begin{equation}\label{taylor}
\Big(\sum\limits_{\kappa =0}^M (i a)^\kappa \, \frac{z^\kappa}{\kappa!}\Big)_{M\geq 0}
\end{equation}
converges in ${\rm Exp}(\C)$  towards $z \mapsto e^{iaz}$ locally uniformly with respect to
$a \in A$, one has
\begin{equation}\label{sect1-eq12A}
\begin{split}
\psi_{\bfK} (a) & = \lim\limits_{M\rightarrow +\infty} \sum\limits_{\kappa =0}^M
\bfK \Big(\frac{(iz)^\kappa}{\kappa!}\Big)\, a^\kappa
\end{split}
\end{equation}
locally uniformly with respect to $a\in A$, which shows that $\psi_{\bfK}$ inherits analyticity. Similarly, when \eqref{taylor} converges in ${\rm Exp}_{<R}(\C)$, in the case
$A \subset ]-R,R[$, then the limit function
\begin{equation}\label{sect1-eq12B}
\begin{split}
\psi_{\bfK_R} (a) & = \lim\limits_{M\rightarrow +\infty} \sum\limits_{\kappa =0}^M
\bfK_R\Big(\frac{(iz)^\kappa}{\kappa!}\Big)\, a^\kappa
\end{split}
\end{equation}
also inherits analyticity.
More precisely, given any element $T$ in the topological $\C$-dual of $\mathscr F$, $T\circ \psi_{\bfK}$ is the restriction to $A$ of an entire function, while $T\circ \psi_{\bfK_\R}$ is the restriction to
$A$ of a holomorphic function in $D_\C (0,R)$. It is easy to see that such setting can be easily modified when $A$ is any open interval of $\R$ with length strictly larger than $2$. If the $(\R,H)$-superoscillating sequence is such that the sequence of entire functions
$$
\Big( z \in \C  \longmapsto \sum\limits_{\nu=0}^{\nu(N)} C_{N,\nu}(N,a) e^{i(a'+ h_{N,\nu}) z}
\Big)_{N\geq 1}
$$
parametrized now by $(a,a') \in \mathbb A$ as in \eqref{sect1-eq8}
converges in ${\rm Exp}(\C)$ (or at least in ${\rm Exp}_{<R}(\C)$ in case
$A \subset ]-R,R[$) to $z\longmapsto \exp (i(a+a')z)$, this time locally uniformly with respect
to both parameters $(a,a')\in \mathbb A$, the continuous maps \eqref{sect1-eq11bis}
satisfy the $({\rm TCSP})_{\mathscr F}$ property on $A$.
Such continuous functions inherit on $A$ the analyticity property
\eqref{sect1-eq12A} in view of the continuity of $\bfK$ (or $\bfK_R$). \\

Let us now consider a specific Cauchy problem for the  Schr\"odinger equation.
 Let $U$ be an open interval in ${\mathbb{R}}$, and consider for $a\in {\mathbb{R}}$
\begin{equation}\label{sect1-eq7}
\Big(i \, \frac{\partial}{\partial t} + \frac{\partial^2}{\partial x^2} - V(x)\Big) \psi(t,x) = 0,\quad
\psi(0,x) = e^{iax},
\end{equation}
where $x\in U\mapsto V(x)$ is a suitable real continuous potential. The most classical examples of open sets $U\subset \R$ which we shall consider are $U= \R$ or $U=\R_{>0}$. When a Green function $K_V~: (x,y,t)\in U\times U \times \R_{>0} \longmapsto  K_V (x,y,t)$ is explicitly known for such Cauchy problem, the evolution of the initial datum
$x \mapsto \psi(0,x) = e^{iax}$, when $a\in \R$, is described (at least formally), by
$$
{\psi_a(t,x)} = \int_U e^{iay}\, K_V(x,y,t)\, dy.
$$
The classical potentials $V$ we have in mind are {\it quadratic}. For example
 $U=\R$, $V(x) = x^2/2$ (quantum harmonic oscillator), see
\cite[\textsection 6.4]{ACSST17A} or also \cite[\textsection 5.4 \& 6]{CSSY22}; $U=\R_{>0}$, $V(x) = \varpi/x^2$ with $\varpi>0$ an absolute constant (centrifugal potential), see \cite[\textsection 5.3]{CSSY22}, or possibly the more general case where $x\mapsto V(t,x)$ is quadratic in $x$ as a polynomial, see \cite{Wolf81}. Under some requirements imposed to the explicit expression of the Green kernel
$K_V$, among them the essential fact that {\it it extends as a function of $y$ to an holomorphic function of $z$ in a single (or double) sector with vertex $0$ about the real axis}, see \cite[Assumption 3.1]{ABCS22}, the formal expression
$$
\int_U e^{iay}\, K_V (x,y,t)\, dy
$$
can be realized concretely as a regularized integral in the Fresnel sense, see
\cite[\textsection 4 \& \textsection 5]{CSSY22} and \cite[\textsection 2 \& \textsection 3]{ABCS22},
which implies that
\begin{equation}\label{2star}
\begin{split}
&  e^{iax}  \longmapsto \int_{U} e^{iay}
K_V(x,y,t)\, dy,\quad (t,x) \in \R_{>0} \times U\ {\rm or} \ \R_{\geq 0}\times U  \\
&  e^{iax} \longmapsto \Big\langle \varphi(t,x)\,,\, \int_{U}\, e^{iay}\, K_V (x,y,t)\, dy \Big\rangle,\quad  \varphi \in \mathcal D(\R_{>0}\times U,\C)\ {\rm or}\  \mathcal D(\R_{\geq 0}\times U,\C)
\end{split}
\end{equation}
are realized in terms of the action of continuous operators from ${\rm Exp}(\C)$ (or ${\rm Exp}_{<R}(\C)$ for some $R>0$) to $\C$. Therefore the maps which are obtained by composition on the right with the map (parametrized as in \eqref{2star} by $(t,x)$ or $\varphi$, considered as elements in $\mathscr{F}'$) with the map which associates to $a\in A \subset \R$ the initial datum $x\mapsto e^{iax}$ {\it inherit automatically real analyticity} besides the properties $({\rm SP})_\C$ or
$({\rm TCSP})_\C$ with respect to particular $(\R,H)$-superoscillating sequences. Such is the case, in particular, when the superoscillating sequence which is involved is the $(\R,H^{\boldsymbol \epsilon})$-superoscillating sequence
defined by \eqref{sect1-eq3} and \eqref{sect1-eq4}.
\\

We conclude this section by discussing an example in which the evolution of $e^{iax}$ can be calculated explicitly in the sense of distributions and yet real analyticity does not follow.
\begin{example}{\rm
We consider the case when $U=\R$ and $V(x) =
\sum_{\ell\in \Z} \widehat V(\ell)\, e^{i\ell x}$ is a smooth non-constant $2\pi$-periodic real potential independent of $t$, the simplest example being $V(x) = \cos (x-\phi)$ with $\phi\in \R$. For each $t>0$ and $\ell\in \Z$, let
$\bfk_{t,\ell}: \ell^\infty_\Z \longrightarrow \ell^\infty_\Z$ be the isometric operator defined by
$$
\bfk_{t,\ell}:\ (u_\kappa)_{\kappa \in \Z}\in \ell^\infty_\Z \longmapsto  \big( -i\, e^{i (\kappa^2 - (\kappa-\ell)^2) t}\, u_{\kappa -\ell}\big)_{\kappa \in \Z} \in \ell^\infty_\Z.
$$
Inspired by the formal computations performed in \cite{GS86}, one can prove, see \cite[\textsection 4]{CSSY21}, that, in the sense of distributions on $\R_{>0}\times \R$,
\begin{equation}\label{sect1-eq13}
\psi_a (t,x) = e^{-i a^2 t} e^{iax}\, \sum\limits_{\kappa \in \Z}
(e^{iat})^{-2\kappa}\, v_{t,\kappa} (e^{ia (\cdot)})\,
e^{-i \kappa^2 t}\, e^{i\kappa x},
\end{equation}
where
\begin{equation}\label{sect1-eq14}
\big(v_{t,\kappa} (e^{ia(\cdot)})\big)_{\kappa \in \Z} = \exp
\Big( \int_0^t \Big( \sum\limits_{\ell \in \Z} \widehat V(\ell)\, (e^{ia\tau})^{2\ell} \, \bfk_{\tau,\ell}\Big)\, d\tau\Big) (\delta_0).
\end{equation}
One can interpret \eqref{sect1-eq13}, when combined with \eqref{sect1-eq14}, as a substitute for
the Green representation formula
\begin{equation}\label{SOLUZvarphi}
 \psi_a(t,x)=\int_{U}e^{iay}\, K_V(x,y,t)dy, \ \ \ x\in U,\ \ t >0,
\end{equation}
for $K_V$ the Green function associated with the periodic potential $V$.
 If $\mathscr F$ denotes the topological space
of complex-valued distributions $\mathcal D'(\R_{>0} \times \R,\C)$ equipped with its strong dual topology, a continuous map $a\mapsto \psi_a$ from $\R$ to $\mathscr F$ is said to be real analytic
(according for example to \cite[Definition 1.5.1 \& Proposition 1.5.2]{OrtW13})
if and only if it extends to a neighborhood $\bfU$ of $\R$ in the complex plane as a continuous function $z \in \bfU \rightarrow \psi_z \in \mathscr F$ such that $z \in \bfU \longmapsto \langle \psi_z, \varphi\rangle$ is holomorphic for any test-function $\varphi \in \mathcal D(\R_{>0} \times \R,\C)$.
Such is not the case for the function
$$
a \in \R \longmapsto \sum\limits_{\kappa \in \Z}
(e^{i a t})^{-2\kappa}\, v_{t,\kappa} (e^{ia (\cdot)})\,
e^{-i \kappa^2 t}\, e^{i\kappa x} \in \mathscr F,
$$
hence for the function $a \in \R \longrightarrow \psi_a$ in view of \eqref{sect1-eq13}.
It may then happen, as it is the case for such class of non-constant smooth $2\pi$-periodic real potentials, that the evolution
$(t,x) \mapsto \psi_a(t,x)$ of the initial datum $x \mapsto e^{iax}$ in
\eqref{sect1-eq7} fails to inherit real analyticity in $a$, when analyticity of distribution-valued continuous functions is understood as above.}
\end{example}

\section{Regular $\C$-supershifts, Bernstein polynomials and Kantorovitch examples}\label{sect2new}

It is convenient in this section to re-scale the real line $\R$ through the affine bijective map $\Upsilon : a \mapsto 2a-1$ from $\R$ to itself that maps $[0,1]$ onto $[-1,1]$. The
$(\R,H)$ superoscillating sequences that we consider are those which correspond to {\it regular sampling} on $[-1,1]$, namely
\begin{equation}\label{sect2new-eq1}
H = H^{\boldsymbol \epsilon} : = \big\{h_{N,\nu}^{\boldsymbol \epsilon} =
1 - 2 \Big( \frac{\nu + \epsilon_N (N-\nu)}{N}\Big)\,:\, N \in \N^*,\ 0 \leq \nu \leq N\big\}
\end{equation}
as in \eqref{sect1-eq3} and
\begin{equation}
\label{sect2new-eq2}
C_{H^{\boldsymbol \epsilon},\nu}(N,a): =
C_{H^{\boldsymbol 0},\nu}(N,a) =
\binom{N}{\nu} \Big( \frac{1+a}{2}\Big)^{N-\nu}
\Big(\frac{1-a}{2}\Big)^\nu,\ 0 \leq \nu \leq N.
\end{equation}
Let us first recall why such sequences are indeed superoscillating.

\begin{proposition}\label{sect2new-prop0}
Let $\boldsymbol \epsilon =(\epsilon_N)_{N\geq 1}$, $\epsilon_N\in [0,1[$, such that
$\lim_{N\rightarrow +\infty} \epsilon_N = 0$ and $a\in \R$.
Then, the sequence of entire functions with exponential growth whose $N$-th element is
\begin{equation}\label{sect2new-prop0-eq1}
T_N^{\boldsymbol \epsilon}[a](z)=\exp (i\, \epsilon_N z)\, \Big( \cos
\Big(z\, \frac{1-\epsilon_N}{N}\Big) + i\, a\, \sin \Big(z\, \frac{1-\epsilon_N}{N}\Big)\Big)^N
\end{equation}
converges in ${\rm Exp}(\C)$ towards the entire function $z\longmapsto e^{iaz}$.
Moreover the convergence is uniform with respect to the parameters $a$ and $\boldsymbol \epsilon$ provided
$a$ remains in a compact $K_\R$ of $\R$ and $\boldsymbol \epsilon$ belongs to a family
$\{\boldsymbol \epsilon_{\iota'} = (\epsilon_{\iota',N})_{N\geq 1}\,:\, \iota' \in I'\}$ of sequences $\boldsymbol
\epsilon_{\iota'} \in ([0,1])^{\N^*}$ which all converge towards $0$, such convergence being uniform with respect to $\iota'$.
Then the $(\R,H^{\boldsymbol \epsilon})$-sequence $\{T_N^{\boldsymbol \epsilon}[a]\,:\, a\in \R\}_{N\geq 1}$ is superoscillating.
\end{proposition}

\begin{proof} Such result is known.
We refer for example to \cite[Lemma 4.1]{CSSY22} with $z$ replaced by $z(1-\epsilon_N)$.
The uniformity properties (with respect to $a\in K_\R$) and to
$\boldsymbol \epsilon$ in a family $\{\boldsymbol \epsilon_{\iota'} = (\epsilon_{\iota',N})_{N\geq 1}\,:\, \iota' \in I'\}$ such as described in the statement are also immediate to check.
\end{proof}

\begin{remark}
Let $A$ be an open interval of $\R$ which contains the segment $[-1,1]$. To say that a continuous map $\psi : A
\rightarrow \C$ satisfies the $({\rm SP})_{\C}$ property on $A$
with respect to the superoscillating sequence $\{T_{N}^{\boldsymbol \epsilon}[a]\,:\, a\in \R\}_{N\geq 1}$, see Definition \ref{sect1-def1} is equivalent to say that the sequence of complex valued continuous functions
\begin{equation}\label{sect2new-eq3}
\Big( b \in B \longmapsto \sum\limits_{\nu=0}^N
\binom{N}{\nu} b^\nu (1-b)^{N-\nu}\, \Psi \Big(\nu\, \frac{1-\epsilon_N}{N}\Big)\Big)_{N\geq 1},
\end{equation}
where
\begin{equation}\label{sect2new-eq3bis}
B := \Upsilon^{-1} (A),\ \Psi = \psi \circ \Upsilon,
\end{equation}
converges to $\Psi$ in $\mathcal C(B,\C)$. When $A$ is an interval of $\R$ with length strictly larger than $2$, to say that a continuous map
$\psi : A \rightarrow \C$ satisfies the $({\rm TCSP})_\C$ property on $A$ with respect to the superoscillating sequence $\{T_{N}^{\boldsymbol \epsilon}[a]\,:\, a\in \R\}_{N\geq 1}$, see Definition \ref{sect1-def2},  is equivalent to say that the sequence of complex valued continuous functions
\begin{equation}\label{sect2new-eq4}
\Big(
(b,b') \in \mathbb B \longmapsto \sum\limits_{\nu=0}^N
\binom{N}{\nu} b^\nu (1-b)^{N-\nu}\, \Psi \Big(b' + \nu\, \frac{1-\epsilon_N}{N}\Big)\Big)_{N\geq 1},
\end{equation}
where
\begin{equation}\label{sect2new-eq4bis}
B = \Upsilon^{-1} (A),\ \mathbb B :=
\Upsilon^{-1} (\mathbb A) = \{(b,b')\in \R \times B\, :\, b'+ [0,1] \subset B,\ b+b'\in B\},\ \Psi = \psi \circ \Upsilon,
\end{equation}
converges to $(b,b') \in \mathbb B \longmapsto \Psi (b+b')$ in
$\mathcal C(\mathbb B,\C)$.
\end{remark}

Let us first recall in our context a well known result relating Serge Bernstein's approximation operators with iterated forward discrete differentiation on the real line.

\begin{proposition}\label{sect2new-prop1}
Let $A$ be an interval of $\R$ with length strictly larger than $2$, $\psi\in \mathcal C(A,\C)$, and
let $\mathbb B$, $\Psi$ be defined as in \eqref{sect2new-eq4bis}. For any sequence
$\boldsymbol \epsilon = (\epsilon_N)_{N\geq 1}$ of elements in $[0,1[$ which tends to $0$ when
$N$ tends to infinity one has, for any $(b,b')\in \mathbb B$ and $N\in \N^*$, that
\begin{equation}\label{sect2new-eq5}
\sum\limits_{\nu=0}^N
\binom{N}{\nu} b^\nu (1-b)^{N-\nu}\, \Psi \Big(b' + \nu\, \frac{1-\epsilon_N}{N}\Big) =
\sum\limits_{\kappa =0}^N
\frac{N!}{(N-\kappa)!}
\frac{(\Delta_{(1-\epsilon_N)/N}^+)^\kappa [\Psi]\,  (b')}{\kappa!}\, b^\kappa,
\end{equation}
where
\begin{equation}\label{sect2new-eq6}
\begin{split}
(\Delta_{(1-\epsilon_N)/N}^+)^0 [\Psi] \, (b') & = \Psi(b') \\
(\Delta_{(1-\epsilon_N)/N}^+)^1 [\Psi] \, (b') & = \Psi(b' + (1-\epsilon_N)/N) - \Psi(b') \\
(\Delta_{(1-\epsilon_N)/N}^+)^2 [\Psi] \, (b') & = (\Delta_{1-\epsilon_N/N}^+)^1[\Psi] (b'+ (1-\epsilon_N)/N)
- (\Delta_{(1-\epsilon_N)/N}^+)^1 [\Psi] \, (b') \\
& = \Psi (b'+ 2(1-\epsilon_N)/N) - 2\, \Psi (b'+ (1-\epsilon_N)/N) + \Psi(b')\\
(\Delta_{(1-\epsilon_N)/N}^+)^3[\Psi] \, (b') & = \cdots
\end{split}
\end{equation}
are the successive forward (+) differences of $\Psi$ at $b'$ (with rate of sampling $(1-\epsilon_N)/N$).
\end{proposition}

\begin{proof} When $\boldsymbol \epsilon = \boldsymbol 0$, the result follows from the classical combinatorial identity
$$
\sum\limits_{\nu=0}^N \binom{N}{\nu} X^\nu (1-X)^{N-\nu} \, c(\nu)  =
\sum\limits_{\nu = 0}^N \binom{N}{\kappa}\, \Delta_{1/N}^\kappa [\,c\,] (0)\, X^\kappa,
$$
see for example \cite[Theorem 7.1.1]{Phil03}. Replacing $\boldsymbol 0$ by $\boldsymbol \epsilon$ amounts just to change the rate $1/N$ used to compute right-divided differences to $(1-\epsilon_N)/N$. For each $b'$ such that $b'+[0,1] \subset B$, it remains to take
$$
c\,(\nu\, (1-\epsilon_N)/N) = c_{b'} (\nu\, (1-\epsilon_N)/N)) = \Psi (b'+ \nu\, (1-\epsilon_N)/N).
$$
Formula \eqref{sect2new-eq5} then follows immediately if one evaluates $X$ at $b$.
\end{proof}

We will also need in this section the following parametric version of another important result due to Serge Bernstein \cite[Theorem A]{Bern35}.

\begin{theorem}\label{sect2new-prop2}
Let $c\in [0,1]$, $\rho > \max (c, 1-c)$,
$\{G_\iota\,:\, \iota \in I\}$ be a bounded family in the space
$\mathcal{H}(D(c,\rho))$ of holomorphic functions in $D_\C(c,\rho)$ and let
$\{\boldsymbol \epsilon_{\iota'} = (\epsilon_{\iota',N})_{N\geq 1}\,:\, \iota' \in I'\}$ be a family of sequences $\boldsymbol
\epsilon_{\iota'} \in ([0,1])^{\N^*}$ which all converge to $0$ uniformly with respect to $\iota'$.
Then, for any $\iota \in I$, $\iota'\in I'$, the sequence of polynomial maps
\begin{equation}\label{sect2new-eq7}
\Big( z \in D_\C(c,\rho)
\longmapsto \sum\limits_{\nu=0}^N
\binom{N}{\nu} z^\nu (1-z)^{N-\nu}\, G_\iota \Big(
\nu\, \frac{1-\epsilon_{\iota',N}}{N}\Big)\Big)_{N\geq 1}
\end{equation}
converges to $G_\iota (z)$ in $\mathcal{H}(D_\C(c,\rho))$. Moreover the convergence on every compact subset of $D_\C(c,\rho)$ is uniform with respect to the parameters
$\iota$ and $\iota'$.
\end{theorem}

\begin{proof}
We follow
the proof presented in \cite[Theorem 4.1.1]{Lor53}, slightly adapting it in order to incorporate the parametric setting. Let $\rho_0 < \rho$ and, for each $\iota \in I$,  let $G_\iota (z) = \sum_{\kappa = 0}^\infty \gamma_{\iota,\kappa} (z-c)^\kappa$, with $\gamma_{\iota,\kappa} \in \C$ the Taylor coefficients of $G_\iota$ about $c$. It follows from the fact that the family $G = \{G_\iota\,:\, \iota \in I\}$ is bounded in
$\mathcal{H}(D_\C(c,\rho))$ that there is a positive constant $M = M_G(\rho_0)$ such that
\begin{equation}\label{prop2-2-eq1}
\forall\, \iota \in I,\ \sum_{\kappa =0}^\infty |\gamma_{\iota,\kappa}|\, \rho_0^\kappa \leq M.
\end{equation}
For each $\boldsymbol \epsilon \in
[0,1[^{\N^*}$ which tends to $0$ when $N$ tends to infinity,
for each $N\in \N^*$ and $\kappa \in \N$, let
$$
B_{N,\kappa}^{\boldsymbol \epsilon} (X) = \sum\limits_{\nu=0}^N \binom{N}{\nu}
\Big( \frac{1-\epsilon_N}{N} \nu -c\Big)^\kappa\,  X^\nu (1-X)^{N-\nu} \in \R[X].
$$
Let us introduce the (for the time being formal) generating function of the sequence of polynomial maps
$(B_{N,\kappa}^{\boldsymbol \epsilon}(z))_{\kappa \geq 0}$, namely
\begin{multline}\label{prop2-2-eq2}
(u,z) \longmapsto \sum\limits_{\kappa =0}^\infty B_{N,\kappa}^{\boldsymbol \epsilon} (z)\,
\frac{u^\kappa}{\kappa!} =
\sum\limits_{\nu=0}^N
\binom{N}{\nu}
z^\nu (1-z)^{N-\nu} \Big(\sum\limits_{\kappa = 0}^\infty
\Big(
\frac{1- \epsilon_N}{N} \, \nu -c\Big)^\kappa
\frac{u^\kappa}{\kappa!}\Big) \\
= \sum\limits_{\nu=0}^N
\binom{N}{\nu}
z^\nu (1-z)^{N-\nu} \Big(\exp
\Big( (1-c)\, \frac{1-\epsilon_N}{N}\Big)\Big)^\nu
\Big( \exp \Big( - c \, \frac{1-\epsilon_N}{N}\, u\Big)\Big)^{N-\nu} \\
= \Big( z \, \exp \Big( (1- c)\, \frac{1-\epsilon_N}{N}\, u\Big) +
(1-z)\, \exp \Big( - c\, \frac{1-\epsilon_N}{N}\, u\Big)\Big)^N \\
= \Big( \sum_{\kappa = 0}^\infty
\big( z (1-c)^\kappa + (1-z) (-c)^\kappa\big)
\, \Big( \frac{1-\epsilon_N}{N}\Big)^\kappa\, \frac{u^\kappa}{\kappa!}\Big)^N.
\end{multline}
As proved in \cite[(6), page 89]{Lor53}, we have  that
\begin{equation}\label{prop2-2-eq3}
\big| z (1-c)^\kappa + (1-z) (-c)^\kappa\big| \leq |z-c|^\kappa,
\end{equation}
for any $z$ such that $|z-c|\geq \min (c,1-c)$ and for any $\kappa \in \N$.
Therefore, for $z$ such that $|z-c| \geq \min (c,1-c)$, the coefficients of the (formal) power series in
$u$ (raised to the power $N$ in the last line of \eqref{prop2-2-eq2}) are bounded from above in absolute value by those of
$$
u \longmapsto \exp \Big( \rho_0 \, \frac{1-\epsilon_N}{N}\, u\Big).
$$
The moduli of the coefficients in the formal expansion in the powers $u^\kappa$ of
$$
(z,u) \longmapsto \sum\limits_{\kappa =0}^\infty B_{N,\kappa}^{\boldsymbol \epsilon} (z)\,
\frac{u^\kappa}{\kappa!}
$$
do not then exceed those of the corresponding coefficients in the expansion in the powers
$u^\kappa$ of
$$
u \longmapsto \exp \big( \rho_0 \, (1-\epsilon_N)\, u\big).
$$
For any $\iota \in I$, $\iota'\in I'$, $N\in \N^*$ and $z\in \overline{D_\C(c,\rho_0)}$, one has
\begin{multline*}
\sum\limits_{\nu=0}^N
\binom{N}{\nu} z^\nu (1-z)^{N-\nu}
G_\iota \Big( \nu \, \frac{1-\epsilon_{\iota',N}}{N}\Big) =
\sum\limits_{\nu=0}^N
\binom{N}{\nu} z^\nu (1-z)^{N-\nu}
\Big( \sum\limits_{\kappa =0}^\infty
\gamma_{\iota,\kappa}\, \Big(
\nu\, \frac{1-\epsilon_{\iota',N}}{N} - c\Big)^\kappa\Big) \\
= \sum\limits_{\kappa =0}^\infty
\gamma_{\iota,\kappa}
\, \Big(\sum\limits_{\nu=0}^N \binom{N}{\nu}\,
\Big(\nu\, \frac{1-\epsilon_{\iota',N}}{N} - c\Big)^{\kappa}
z^\nu (1-z)^{N-\nu}\Big).
\end{multline*}
Since
$$
\Big|\sum\limits_{\nu=0}^N \binom{N}{\nu}\,
\Big(\nu\, \frac{1-\epsilon_{\iota',N}}{N} - c\Big)^{\kappa}
z^\nu (1-z)^{N-\nu}\Big| = |B_{N,k}^{\boldsymbol \epsilon_\iota}(z)| \leq (\rho_0 (1-\epsilon_N))^k ,
$$
according to the reasoning above, the sequence of polynomial functions
$$
\Big(\sum\limits_{\nu=0}^N
\binom{N}{\nu} z^\nu (1-z)^{N-\nu}
G_\iota \Big( \nu \, \frac{1-\epsilon_{\iota',N}}{N}\Big)\Big)_{N\geq 1}
$$
is uniformly bounded in $D_\C(c,\rho)$ with respect both to
$\iota \in I$ and $\iota'\in I'$. Since the family
$\{G_\iota\,;\, \iota \in I\}$ is bounded in $\mathcal{H}(V)$, where $V$ is any open neighborhood of
$[0,1]$, the family $\{(G_{\iota})_{|[0,1]}\,:\, \iota \in I\}$ admits a uniform modulus of continuity
$[0,1]$ Classical facts about Bernstein's approximation of continuous functions on $[0,1]$, see for example \cite[\textsection 1.6]{Lor53}, ensure that for any $\iota \in I$ and
$\iota'\in I'$, the sequence of functions
$$
\Big(\sum\limits_{\nu=0}^N
\binom{N}{\nu} z^\nu (1-z)^{N-\nu}
G_\iota \Big( \nu \, \frac{1-\epsilon_{\iota',N}}{N}\Big)\Big)_{N\geq 1}
$$
converges in $\mathcal C([0,1],\C)$ to $z\in [0,1] \longmapsto G_\iota (z)$, such convergence being uniform with respect to the choice of $\iota \in I$ or $\iota'\in I'$ as soon as the convergence of $(\epsilon_{\iota',N})_{N\geq 1}$ towards $0$ is uniform with respect to the index $\iota'$.
Since all points in $[0,1]$ as a subset of $D_\C(c,\rho)$ are accumulation points, the assertion of the proposition follows then from Vitali-Porter theorem, see
\cite[p. 68]{Titch39} or also \cite[\textsection 2.4]{Schiff93}.
\end{proof}

Theorem \ref{sect2new-prop2} implies the following result.

\begin{theorem}\label{sect2new-thm1} Let $A$ be an open interval of $\R$ with length $R$ strictly larger than $2$ and let
$F$ be either an entire function when $A$ is unbounded, or a holomorphic function on
\begin{equation}\label{sect2new-eq7bis}
W_A = \bigcup_{\{a'\in A\,:\, a'+[-1,1]\subset A\}}
\overline{D_{\C}\big(a' + C(a'), {\rm dist} \big(a'+ C(a'),\R \setminus A\big)\big)},
\end{equation}
where
\begin{equation}\label{sect2new-eq7ter}
C(a') = 2\, \Big(\frac{a'-1 - \inf A}{R-2}-1\Big)  \in ]-1,1[
\end{equation}
when $A$ is bounded. Then $\psi= F_{|A}$ satisfies on $A$ the $({\rm TCSP})_\C$ property
with respect to the superoscillating sequence $\{T_{N}^{\boldsymbol \epsilon}[a]\,:\, a\in \R\}_{N\geq 1}$, where $\boldsymbol \epsilon$ is any sequence
of elements in $[0,1[$ which tends to $0$. Moreover, given a family of sequences $\boldsymbol
\epsilon_{\iota'} \in ([0,1])^{\N^*}$, $\iota'\in I'$, which all converge to $0$ uniformly with respect to $\iota'$ on any compact subset of the
all sequences of functions
$$
\Big(
(a,a') \in \mathbb A \longmapsto
\sum\limits_{\nu=0}^N
\binom{N}{\nu}
\Big(\frac{1+a}{2}\Big)^{N-\nu}
\Big(\frac{1-a}{2}\Big)^\nu \, \psi(a' + h_{N,\nu}^{\boldsymbol \epsilon_{\iota'}})\Big)_{N\geq 1}
$$
to $(a,a') \in \mathbb A \longmapsto \psi (a+a')$ is also uniform with respect to the index $\iota'$.
\end{theorem}

\begin{proof} Consider first the case where $A$ is a bounded open interval
with length $R>2$. Then $B= \Upsilon^{-1} (A) = ]b_{\rm min}, b_{\rm max}[$ with
$b_{\rm max}-b_{\rm min} = R/2>1$. Let $V$ be an open neighborhood of
$$
W_{B} =
\bigcup_{b_{\rm min}< b' < b_{\rm max} -1}
\overline{D_{\C}\big(b'+c(b'), \rho_0(b')}\big)
$$
where
\begin{equation}\label{sect2new-eq8}
c(b') = 1 - \frac{b'- b_{\rm min}}{R/2-1} \in ]0,1[,\quad
\rho_0(b') = {\rm dist} \big(b'+ c (b'), \R \setminus \Upsilon^{-1} (A)\big),
\end{equation}
so that
$$
\lim_{b' \searrow b_{\rm min}}
c (b') = 1,\quad \lim_{b' \nearrow b_{\rm max}}c (b') = 0,\
\lim_{b'\searrow b_{\rm min}} \rho_0(b') =
\lim_{b'\nearrow b_{\rm min}} \rho_0(b') = \max (1, R/2 -1).
$$
Suppose that $\Psi = \psi\circ \Upsilon$ is the restriction to
the open interval $\Upsilon^{-1} (A)$ of a holomorphic function in $V$,
which amounts to say that $\psi$ is the restriction to $A$ of a holomorphic function
in the open subset $2 V -1$ of the complex plane.
This implies that for any $b'\in ]b_{\rm min}, b_{\rm max}-1[$, there exists a holomorphic function
$G_{b'}$ in an open neighborhood $D_\C (c(b'),\rho(b'))$ of $\overline {D_\C(c(b'), \rho_0(b'))}$ such that $G_{b'}= \Psi(b'+\cdot)$ in an open neighborhood of $[0,1]$. It follows then from
Theorem \ref{sect2new-prop2} with
$$
\{G_\iota\,:\, \iota \in I\}
= \{z\mapsto G_{b'\pm \eta} (z) \,:\, 0\leq \eta \leq \eta (b')\},
$$
and from the definition of $\rho_0(b')< \rho(b')$
that, given $\{\boldsymbol \epsilon_{\iota'} = (\epsilon_{\iota',N})_{N\geq 1}\,:\, \iota' \in I'\}$ a  family of sequences $\boldsymbol
\epsilon_{\iota'} \in ([0,1])^{\N^*}$ which converge towards $0$ uniformly with respect to $\iota'$,
the sequence of polynomial functions
\begin{multline*}
\Big(b \in ]b_{\rm min}-b',b_{\rm max}-b'[ \
\longmapsto \sum\limits_{\nu =0}^N
\binom{N}{\nu}\, b^\nu (1-b)^{N-\nu} G_{b'\pm \eta} \Big( \nu \frac{1-\epsilon_{\iota',N}}{N}\Big)\Big)_{N\geq 1} \\
=  \Big(b \in ]b_{\rm min} - b', b_{\rm max} -b'[ \ \longmapsto \sum\limits_{\nu =0}^N
\binom{N}{\nu}\, b^\nu (1-b)^{N-\nu} \Psi \Big(b' \pm \eta +  \nu \frac{1-\epsilon_{\iota',N}}{N}\Big)\Big)_{N\geq 1}
\end{multline*}
converges uniformly on any compact subset of in $\Upsilon^{-1} (A) - b'$ towards
$G_{b'\pm \eta} (b) = \Psi (b+b')$,
the convergence being uniform with respect both to $\pm \eta$ such that $0\leq \eta
\leq \eta (b')$ and to $\iota'$.
Since
\begin{multline*}
2 W_B -1 = \bigcup_{\{a'\in A\,;\, a'+ [-1,1] \subset A\}}
\overline {D_\C \Big( (a'-1) + 2 \Big( \frac{a'-1 - \inf A}{R-2}\Big), 2\rho_0 \Big(
\frac{a'}{2}\Big)\Big)} \\
= \bigcup_{\{a'\in A\,;\, a'+ [-1,1] \subset A\}}
\overline {D_\C \big( a'+C(a'),
{\rm dist}
\, \big( a'+C(a'), \R \setminus A\big)
\big)} = W_A
\end{multline*}
according to the definition of $a'\mapsto C(a')$, Theorem \ref{sect2new-thm1} follows in the case where $A$ is bounded.
When $A$ is unbounded, one may exhaust $A$ with an increasing sequence of bounded intervals $(A_k)_{k\geq 1}$ whose sequence of lengths $(R_k)_{k\geq 1}$ tends to infinity.
Theorem \ref{sect2new-thm1} in the unbounded case follows then immediately from Theorem \ref{sect2new-thm1} in the bounded case since testing $({\rm TCSP})_\C$
with respect to $(\R,H^{\boldsymbol \epsilon})$
on $A$ amounts clearly to test such property on each
$A_k$.
\end{proof}

\begin{remark}\label{sect2new-rem1}
When $A$ is bounded with length $R>2$, one has
\begin{equation}\label{sect2new-eq9}
W_A \subset \{z\in \C\,:\, {\rm dist}(z,\overline A) \leq \max (2, R-2)\}.
\end{equation}
Any holomorphic function in an open neighborhood of $\{z\in \C\,:\, {\rm dist}(z,\overline A) \leq \max (2, R-2)\}$ is then such that its restriction to $A$ inherits the $({\rm TCSP})_\C$ property
with respect to the superoscillating sequence $\{T_{N}^{\boldsymbol \epsilon}[a]\,:\, a\in \R\}_{N\geq 1}$, where $\boldsymbol \epsilon \in [0,1[^{\N^*}$ tends to
$0$ when $N$ tends to infinity (with the uniformity clause with respect to the choice of such sequence $\boldsymbol \epsilon$).
\end{remark}

Theorem \ref{sect2new-thm1} suggests the introduction of the following definition.

\begin{definition}\label{sect2new-def1} (Regular $\C$-supershift)
Let $A$ be an open interval of $\R$ with length $R$ strictly larger than $2$. A continuous map
$\psi: A \rightarrow \C$ is called a {\it regular $\C$-supershift} if the following two conditions are fulfilled.
\begin{enumerate}
\item The map $\psi$ satisfies the $({\rm TCSP})_\C$ property on $A$ with respect to any superoscillating sequence $\{T_{N}^{\boldsymbol \epsilon}[a]\,:\, a\in \R\}_{N\geq 1}$, as in
\eqref{sect1-eq3} and \eqref{sect1-eq4}, where $\boldsymbol \epsilon$ is any sequence in
$[0,1[^{\N^*}$ which tends to $0$ when $N$ tends to infinity.
\item Given a family $\{(\epsilon_{\iota',N})_{N\geq 1}\,:\, \iota \in I'\}$  of such sequences, such that the convergence of all sequences $\boldsymbol \epsilon_{\iota'}$ towards $0$ is uniform with respect to the index $\iota'$, the convergence  for $(a,a') \in \mathbb A$ of the polynomial functions
\begin{multline}\label{sect2new-eq10}
 \sum\limits_{\nu=0}^{N}
\binom{N}{\nu}
\Big(\frac{1+a}{2}\Big)^{N-\nu}
\Big(\frac{1-a}{2}\Big)^\nu
\psi\Big( a' +
\Big(1 - 2\, \Big(\frac{\nu + \epsilon_{\iota',N}\, (N-\nu)}{N}\Big)\Big)\Big),\quad N=1,2,...
\end{multline}
to $\psi(a+a')$ in $\mathcal C_c(\mathbb A,\C)$ is uniform with respect to the index $\iota'$.
\end{enumerate}
\end{definition}

It follows from from Theorem \ref{sect2new-thm1} that such theorem that real analytic functions, provided they can be extended as holomorphic functions in sufficiently large complex domains containing $A$, provide examples of regular supershifts in $A$. Such is the case for example for restrictions to $A$ of entire functions in case $A$ is unbounded. The natural question is to ask whether  regular $\C$-supershifts inherit some analyticity property. We will deduce from a result of Leonid Kantorovich \cite{Kanto31} that the answer to such question is NO in the case of bounded intervals $A$. \\

The following theorem follows from a result in \cite{Kanto31}, see also
\cite[Theorem 4.1.3 \& subsequent remark]{Lor53}. We modified slightly the construction here in order to formulate a parametric version, as we did with Bernstein's theorem
in Theorem \ref{sect2new-prop2}.

\begin{theorem}\label{sect2new-prop3} Let $G^-$ and $G^+$ two entire functions such that
\begin{equation}\label{sect2new-eq11}
G^{-} (1/2) = G^+(1/2),\quad \frac{dG^-}{dz}(1/2) \not= \frac{dG^+}{dz}(1/2)
\end{equation}
and let $g^{\pm} : b \in \R \rightarrow \C$ be the continuous map defined by
\begin{equation}\label{sect2new-eq12}
g^{\pm} (b) = \begin{cases} G^-(b)\ & {\rm if}\ b < 1/2 \\
G^{+} (b)\ & {\rm if}\ b \geq 1/2.
\end{cases}
\end{equation}
Let $0 < \eta < 1/2$. There exists a neighborhood $V^-_{\eta}$ of $0$ in $\C$ and a neighborhood
$V^+_{\eta}$ of $1$ in $\C$ such that, for any $b'\in [-\eta,\eta]$,
for any $\boldsymbol \epsilon \in [0,1[^{\N^*}$ which tends to $0$ at infinity,
the sequence of polynomial maps
\begin{equation}\label{sect2new-eq12A}
\Big(z \longmapsto \sum\limits_{\nu=0}^N
\binom{N}{\nu} z^\nu (1-z)^{N-\nu} g^{\pm} \Big(b' + \nu \, \frac{1-\epsilon_N}{N}\Big)\Big)_{N\geq 1}
\end{equation}
converges uniformly on any compact subset of $V_\eta^-$ to $z\longmapsto G^- (z+b')$ and uniformly on any compact subset of $V^+_\eta$ to $z\longmapsto G^+ (z+b')$. The convergence is uniform with respect to the parameter $b'\in [-\eta,\eta]$. Moreover, if we consider a family
$\{(\epsilon_{\iota',N})_{N\geq 1}\,:\, \iota'\in I'\}$, where all sequences $\boldsymbol \epsilon_{\iota'}$ tend to $0$ at infinity uniformly with respect to $\iota'$, the above convergences
with $\boldsymbol \epsilon_{\iota'}$ instead of $\boldsymbol \epsilon$, respectively on compact subsets of $V_{\eta}^-$ and $V_\eta^+$, are also uniform with respect to the index $\iota'$.
\end{theorem}

\begin{proof}
For each $c\in [1/2-\eta,1/2+\eta]$, let $\mathscr L_c$ be the lemniscate
$$
\mathscr L_c = \{z\in \C\,:\, |z|^{c} |1-z|^{1-c} = c^c (1-c)^{1-c}\},
$$
$V^-(c)$ (respectively $V^+(c)$) be the interior of its left loop (respectively its right loop) about $0$ (respectively $1$),
$$V^-_\eta = \bigcap_{c\in [1/2-\eta,1/2+\eta]} V^-(c),\qquad {\rm and}\qquad
V^+_\eta= \bigcap_{c\in [1/2 -\eta,1/2 + \eta]} V^+(c).$$
Fix also $\{(\epsilon_{\iota',N})_{N\geq 1}\,:\, \iota'\in I'\}$ such that all sequences $\boldsymbol \epsilon_{\iota'}$ tend to $0$ at infinity uniformly with respect to $\iota'$.
It follows from Theorem \ref{sect2new-prop2} that the two sequences of polynomial maps
\begin{equation*}
\begin{split}
& \Big(z \longmapsto \sum\limits_{\nu=0}^N
\binom{N}{\nu} z^\nu (1-z)^{N-\nu} G^- \Big(b' + \nu \, \frac{1-\epsilon_{\iota',N}}{N}\Big)\Big)_{N\geq 1} \\
& \Big(z \longmapsto \sum\limits_{\nu=0}^N
\binom{N}{\nu} z^\nu (1-z)^{N-\nu} G^+ \Big(b' + \nu \, \frac{1-\epsilon_{\iota',N}}{N}\Big)\Big)_{N\geq 1}
\end{split}
\end{equation*}
converge on any compact subset of the complex plane to respectively
$G^- (b'+z)$ and $G^+(b'+z)$, uniformly
both with respect to $b'\in [-\eta,\eta]$ and the index $\iota'$.
Fix a compact $K^+ \subset V^+_\eta$. Since
$K^+$ is contained in each $V^+(c)$ for $c\in [1/2-\eta,1/2 + \eta]$, there is a constant $q_\eta(K^+)<1$ such that for any $c\in [1/2-\eta,1/2 + \eta]$,
\begin{equation}\label{sect2new-eq13}
\begin{split}
z \in K^+ & \Longrightarrow
\Big(\frac{|z|}{c}\Big)^c
\Big(\frac{|1-z|}{1-c}\Big)^{1-c} \leq q_\eta(K^+) < 1 \\
& \Longrightarrow \frac{c}{|z|}\, \frac{|1 - z|}{1-c} \leq \Big( q_\eta(K^+)\, \frac{c}{|z|}\Big) ^{1/(1-c)} \leq 1,
\end{split}
\end{equation}
see the proof of \cite[Theorem 4.3.1]{Lor53}.
Then one has for each $z\in K^+$, for each $b'\in [-\eta,\eta]$ and $H=G^+$ or $H=G^-$ that
\begin{multline}\label{sect2new-eq14}
\Big|\sum\limits_{\{\nu \in \N\,:
\, \nu \leq (1/2-b') N \}} \binom{N}{\nu} z^\nu (1-z)^{N-\nu}
H \Big(b' + \nu\, \Big(\frac{1- \epsilon_{\iota',N}}{N}\Big)\Big)\Big|
\\
\leq \sup_{[-b',1+b']} |H|
\, \sum\limits_{\{\nu \in \N\,:
\, \nu \leq (1/2-b')\, N \}}
\binom{N}{\nu} \, (1/2-b')^\nu (1/2+b')^{N-\nu}
\Big(\frac{|z|}{1/2-b'}\Big)^\nu
\Big(\frac{|1-z|}{1/2+b'}\Big)^{N-\nu}  \\
\leq \sup_{[-b',1+b']} |H| \, \sum\limits_{\{\nu \in \N\,:
\, \nu \leq (1/2-b') N \}}\Big(\frac{|z|}{1/2-b'}\Big)^\nu
\Big(\frac{|1-z|}{1/2+b'}\Big)^{N-\nu}.
\end{multline}
Since for any $\nu \in \N$ such that $\nu \leq (1/2-b') N$, one has for $z\in K^+$ that
\begin{multline*}
\Big(\frac{|z|}{1/2-b'}\Big)^\nu
\Big(\frac{|1-z|}{1/2+b'}\Big)^{N-\nu} \\
= \Big(\Big(\frac{|z|}{1/2-b'}\Big)^{1/2-b'}
\Big(\frac{|1-z|}{1/2+b'}\Big)^{1/2+b'}\Big)^N \, \Big(
\frac{1/2-b'}{|z|}\, \frac{|1-z|}{1/2+b'}\Big)^{N(1/2-b') -\nu} \leq (q_\eta(K^+))^N
\end{multline*}
according to \eqref{sect2new-eq13}, one concludes that for any $N\in \N^*$ and
$b'\in [-\eta,\eta]$, one has from \eqref{sect2new-eq14} that
\begin{multline}\label{sect2new-eq15}
\sup_{z\in K^+}
\Big|\sum\limits_{\{\nu \in \N\,:
\, \nu \leq (1/2-b') N \}} \binom{N}{\nu} z^\nu (1-z)^{N-\nu}
H \Big(b' + \nu\, \Big(\frac{1- \epsilon_{\iota',N}}{N}\Big)\Big)\Big| \\
\leq (\sup_{[-\eta,1+\eta]} |H|)\, N\, (q_\eta(K^+))^N,
\end{multline}
which means that the sequence of polynomial maps
$$
\Big(z \longmapsto \sum\limits_{\{\nu \in \N\,:
\, \nu \leq (1/2-b') N \}} \binom{N}{\nu} z^\nu (1-z)^{N-\nu}
H \Big(b' + \nu\, \Big(\frac{1- \epsilon_{\iota',N}}{N}\Big)\Big)\Big)_{N\geq 1}
$$
converges uniformly towards $0$ on $K^+$, the convergence being uniform with respect to
$b'\in [-\eta,\eta]$ and to the index $\iota'$. Since $H$ can be either
$G^-$ or $G^+$ and the sequence of polynomial maps
$$
\Big(z \longmapsto \sum\limits_{\nu=0}^N
\binom{N}{\nu} z^\nu (1-z)^{N-\nu} G^+ \Big(b' + \nu \, \frac{1-\epsilon_{\iota',N}}{N}\Big)\Big)_{N\geq 1}
$$
converges uniformly on $K^+$ towards $z\mapsto G^+(b'+z)$,
the convergence being uniform both with respect to $b'\in [-\eta,\eta]$ and the index $\iota'$,
one concludes that the sequence of polynomial maps \eqref{sect2new-eq12A}
converges uniformly on $K^+$ towards $z \longmapsto G^+(b'+z)$, the convergence being uniform both with respect to $b'\in [-\eta,\eta]$ and the index $\iota'$.
A similar reasoning shows that if $K^- \subset V_\eta^-$, the sequence
\eqref{sect2new-eq12A} converges uniformly on $K^-$ to $z \longmapsto G^-(b'+z)$ uniformly both with respect to $b'\in [-\eta,\eta]$ and the index $\iota'$.
\end{proof}

We now prove the following lemma.

\begin{lemma}\label{sect2new-lem1}
Let $A$ be an open interval of $\R$ with length strictly larger than $2$ and
$\psi : A \rightarrow \C$ be a regular supershift in $A$. Let $\theta$ be a smooth test-function with support in $[0,\epsilon]$, where $\epsilon < R-2$ and $A_\epsilon =
\{a \in A\,;\, a-\epsilon \in A\}$. Then the function
$$
\psi * \theta : a \in A_\epsilon \longmapsto \int_0^{\epsilon}
\theta(\alpha)\, \psi(a-\alpha)\, d\alpha = \int_\R \theta(\alpha)\, \psi(a-\alpha)\, d\alpha
$$
is a smooth regular $\C$-supershift in $A_\epsilon$.
\end{lemma}

\begin{proof}
Let
$$
\mathbb A_\epsilon :=
\{(a,a') \in \R \times A_\epsilon\,:\, a'+[-1,1] \subset A_\epsilon,\ a+a'\in A_\epsilon\}
$$
and $\{(\epsilon_{\iota',N})_{N\geq 1}\,:\, \iota'\in I'\}$ be a family of elements in
$[0,1[^{\N^*}$ which all tend to $0$ at infinity, uniformly with respect to $\iota'$.
For any $(a_0,a'_0)\in \mathbb A_\epsilon$ and any $N\in \N^*$, one has for
$(a,a')$ in a sufficiently small neighborhood $U_{a_0,a'_0}$ of $(a_0,a'_0)$
\begin{multline}\label{sect2new-eq16}
\int_0^\epsilon \theta(\alpha)\,
\Big(\sum\limits_{\nu=0}^N
\binom{N}{\nu} \Big(
\frac{1+a}{2}\Big)^{N-\nu}
\Big(\frac{1-a}{2}\Big)^\nu \psi\Big(a'-\alpha +
\nu \, \frac{1-\epsilon_{\iota',N}}{N}\Big)\Big)\, d\alpha  \\
= \sum\limits_{\nu=0}^N \binom{N}{\nu} \Big(
\frac{1+a}{2}\Big)^{N-\nu}
\Big(\frac{1-a}{2}\Big)^\nu (\psi*\theta)\Big( a' + \nu \, \frac{1-\epsilon_{\iota',N}}{N}\Big).
\end{multline}
Since $\psi$ is a regular $\C$-supershift on $A$, the sequence of functions
\begin{equation*}
\Big( (a,a') \longmapsto \sum\limits_{\nu=0}^N
\binom{N}{\nu} \Big(
\frac{1+a}{2}\Big)^{N-\nu}
\Big(\frac{1-a}{2}\Big)^\nu \psi\Big(a'-\alpha +
\nu \, \frac{1-\epsilon_{\iota',N}}{N}\Big)\Big)_{N\geq 1},
\end{equation*}
considered as indexed by the parameter $\alpha \in [0,\epsilon]$,
converges uniformly on any compact subset of $U_{a_0,a'_0}$ towards
$$
(a,a') \longmapsto \psi (a+ a'-\alpha),
$$
the convergence being also uniform both with respect to $\alpha \in [0,\epsilon]$ and to the index
$\iota'$. As a consequence of \eqref{sect2new-eq16}, the sequence of functions
$$
\Big( (a,a') \longmapsto \sum\limits_{\nu=0}^N \binom{N}{\nu} \Big(
\frac{1+a}{2}\Big)^{N-\nu}
\Big(\frac{1-a}{2}\Big)^\nu (\psi*\theta)\Big( a' + \nu \, \frac{1-\epsilon_{\iota',N}}{N}\Big)\Big)_{N\geq 1}
$$
converges uniformly on any such compact subset of $U_{a_0,a'_0}$ towards
$$
(a,a') \longmapsto \int_0^\epsilon \theta (\epsilon) \psi (a+a'-\alpha)\, d\alpha =
(\psi * \theta) (a+a').
$$
This shows that $\psi* \theta$ is a regular $\C$-supershift in $\mathbb A_\epsilon$.
It is smooth since it inherits the smoothness of the regularizing test-function $\theta$.
\end{proof}

We deduce then from Theorem \ref{sect2new-prop3} the following result.

\begin{theorem}\label{sect2new-thm2}
Let $g^{\pm}$ be the continuous complex valued function with $1/2$ as single non-real analyticity point which is defined
by \eqref{sect2new-eq11} and \eqref{sect2new-eq12}.
There exists an open interval $A_0=]-1-\rho_0,1+\rho_0[\subset \R$ containing $[-1,1]$ such that $a \in A_0 \longmapsto \psi_0(a) = g^{\pm} ((1+a)/2)$ is a regular $\C$-supershift in $A_0$.
As a consequence, the fact that $\psi$ is a smooth regular $\C$-supershift on some open interval
$A\subset \R$ with diameter strictly larger than $2$ does not imply in general that $\psi$ is real analytic on $A$.
\end{theorem}

\begin{proof}
The fact that the first assertion implies the second one is an immediate consequence of Lemma
\ref{sect2new-lem1}. One can choose two entire functions $G^-$ and $G^+$ such that
\eqref{sect2new-eq11} is fulfilled and a smooth regularization function $\theta$
such that $\psi = \psi_0 * \theta$ shares with $\psi_0$ the following property, namely that there exists at least one point of non-real analyticity in $]-1,1[$.
It remains then to prove the first assertion. Let $\eta>0$, $V_\eta^-$ and $V_\eta^+$ as in Theorem
\ref{sect2new-prop3}. Since the family of continuous functions
$$
\big\{x \in [0,1] \longmapsto g^{\pm} (b'+x)\,:\, b'\in [-\eta,\eta]\}
$$
is uniformly Lipschitz, hence admits a uniform modulus of continuity
$[0,1]$, classical facts about Bernstein's approximation of continuous functions on $[0,1]$, see for example \cite[\textsection 1.6]{Lor53}, imply that for each $b'\in [-\eta,\eta]$, the sequence of polynomial functions
$$
\Big( b\in [0,1] \longmapsto \sum\limits_{\nu=0}^N
\binom{N}{\nu} b^\nu (1-b)^{N-\nu} g^{\pm} \Big(b' + \nu \, \frac{1-\epsilon_N}{N}\Big)\Big)_{N\geq 1}
$$
converges uniformly on $[0,1]$ towards $b\longmapsto g^{\pm} (b+b')$, the convergence being uniform both with respect to $b'\in [-\eta,\eta]$ and to the index $\iota'$. If we choose $\epsilon_\eta \in ]0,\eta]$
such that $[-\epsilon_\eta,\epsilon_\eta] \subset V_\eta^{-}$ and $[1-\epsilon_\eta,1+\epsilon_\eta]
\subset V_\eta^+$, Theorem \ref{sect2new-prop3} shows that the same result holds on any compact subset of $]-\epsilon_\eta,1+\epsilon_\eta[$. The first assertion in the theorem follows if one makes the rescaling $b = (a+1)/2$.
\end{proof}

From Theorem \ref{sect2new-thm2} it is immediate the following statement which follows taking for example
$G^-(z) = 1-2z$ and $G^+(z) = 2z-1$, so that $g^{\pm} (x)=|x|$.
\begin{corollary} There exist smooth regular $\C$-supershifts which are not real analytic.
\end{corollary}
\section{Operations preserving supershift}
Since it is not immediate to establish when a function has the regular $\C$-supershift property,
we conclude this paper by analyzing which operations preserve this property.
The following theorems mention two of them.

\begin{theorem}\label{sect2new-thm3}
Given an open interval $A$ of $\R$ with diameter strictly larger than $2$, the class of regular $\C$-supershifts on $A$ is stable by multiplication by the map $a\in A \longmapsto a\in \C$, hence by any complex valued polynomial map
$a\in A \longmapsto p(a) \in \C$.
\end{theorem}

\begin{proof}
Let $\{(\epsilon_{\iota',N})_{N\geq 1}\,:\, \iota'\in I'\}$ be a family of elements in
$[0,1[^{\N^*}$ which all tend to $0$ at infinity, uniformly with respect to $\iota'$.
Let $B= \Upsilon^{-1}(A)$ and $\Psi = \psi \circ \Upsilon$.
An easy computation shows that if $\Phi$ is the function $\Phi : b\in B \longmapsto b\, \Psi(b)$
then, for any $N\in \N^*$ with $N\geq 2$ and any
$$
(b,b')
\in \mathbb B = \{(b,b')\in \R \times B\,:\,
b'+[0,1]\subset B,\ b+b'\in B\},
$$
one has
\begin{multline}\label{sect2new-eq17}
\sum\limits_{\nu=0}^N
\binom{N}{\nu} b^\nu (1-b)^{N-\nu}
\Phi \Big(b' + \nu\, \frac{1-\epsilon_{\iota',N}}{N}\Big) \\
=
(1-\epsilon_N)\, b
\sum\limits_{\nu=0}^{N-1} \binom{N-1}{\nu}
b^\nu (1-b)^{N-1-\nu} \Psi
\Big(b' + \frac{1-\epsilon_{\iota',N}}{N} +
\nu\, \frac{1-\epsilon_{\iota',N-1}^{[1]}}{N-1}\Big) \\
+ b'\, \sum\limits_{\nu=0}^N
\binom{N}{\nu} b^\nu (1-b)^{N-\nu}
\Psi \Big(b' + \nu\, \frac{1-\epsilon_{\iota',N}}{N}\Big),
\end{multline}
where
\begin{equation}\label{sect2new-eq17A}
1-\epsilon_{\iota',N-1}^{[1]} = \frac{N-1}{N} (1- \epsilon_{\iota',N}).
\end{equation}
Let $K$ be a compact subset of $\mathbb B$ and $\chi_K>0$ such that
$b'+b + \tau \in B$ for any $(b,b')\in K$ and $0\leq \tau \leq \chi_K$.
Since $\psi$ is a regular $\C$-supershift, there exists, for any
$\varepsilon >0$, an integer $N_{K,\varepsilon}$ such that
\begin{multline*}
N \geq N_{K,\varepsilon} \\ \Longrightarrow
\sup_{\stackrel{((b,b'),\tau) \in K \times [0,\chi_K]}{\iota'\in I'}}
\Big|\Psi (b'+ \tau  + b) -  \sum\limits_{\nu=0}^{N-1} \binom{N-1}{\nu}
b^\nu (1-b)^{N-1-\nu} \Psi
\Big(b' + \tau  +
\nu\, \frac{1-\epsilon_{\iota',N-1}^{[1]}}{N-1}\Big)\Big| \leq \varepsilon.
\end{multline*}
One has in particular for $N\geq \max (N_{K,\varepsilon}, 1/\chi_K)$ that
$$
\sup_{\stackrel{(b,b') \in K}{\iota'\in I'}}
\Big|\Psi \Big(b'+ \frac{1-\epsilon_{\iota',N}}{N} + b\Big) -
\sum\limits_{\nu=0}^{N-1} \binom{N-1}{\nu}
b^\nu (1-b)^{N-1-\nu} \Psi
\Big(b' + \frac{1-\epsilon_{\iota',N}}{N}  +
\nu\, \frac{1-\epsilon_{\iota',N-1}^{[1]}}{N-1}\Big)\Big|\leq \varepsilon.
$$
Observe that the sequences of functions
\begin{equation*}
\begin{split}
& \Big( (b,b') \in \mathbb B \longmapsto \Psi \Big(b'+ \frac{1-\epsilon_{\iota',N}}{N} + b\Big)\Big)_{N\geq 1}\\
& \Big( (b,b') \in \mathbb B
\longmapsto \sum\limits_{\nu=0}^N
\binom{N}{\nu} b^\nu (1-b)^{N-\nu}
\Psi \Big(b' + \nu\, \frac{1-\epsilon_{\iota',N}}{N}\Big)\Big)_{N\geq 1}
\end{split}
\end{equation*}
converge both locally uniformly on $\mathbb B$ towards $(b,b') \longmapsto \Psi(b+b')$ on $K$, the convergence being uniform with respect to the index $\iota'$. Comparing the limits when $N$ tends to infinity of both sides of
\eqref{sect2new-eq17}, one gets that the sequence of functions
$$
\Big( (b,b') \in \mathbb B
\longmapsto \sum\limits_{\nu=0}^N
\binom{N}{\nu} b^\nu (1-b)^{N-\nu}
\Phi \Big(b' + \nu\, \frac{1-\epsilon_{\iota',N}}{N}\Big)\Big)_{N\geq 1}
$$
converges locally uniformly on $\mathbb B$ towards $(b,b') \longmapsto \Phi(b+b')$,
the convergence being uniform with respect to the index $\iota'$. After the rescaling
$b = (1+a)/2$, one gets that $a \in A \longmapsto (2a-1) \psi(a)$ is a regular
$\C$-supershift on $a$, so is $a \in A \longmapsto a \, \psi(a)$ since $\psi$ already is. The proof of Theorem \ref{sect2new-thm3} is thus completed.
\end{proof}


\begin{theorem}\label{sect2new-thm4}
Given an open interval $A$ of $\R$ with diameter strictly larger than $2$, the class of regular $\C$-supershifts on $A$ is stable by primitivization, namely, if $\psi : A \rightarrow \C$ is a $\C$-regular supershift and
$a_0 \in A$, then
$$
a \in A \longmapsto \int_{a_0}^a \psi(\alpha)\, d\alpha
$$
is a $\C$-regular supershift.
\end{theorem}

\begin{proof}
Let $\psi$ be a regular $\C$-supershift in $A$, $B=\Upsilon^{-1}(A)$ and $\Psi = \psi\circ \Upsilon$, where, as usual,
$\Upsilon (a)=2a-1\in \R$.
Let $\{(\epsilon_{\iota',N})_{N\geq 1}\,:\, \iota'\in I'\}$ be a family of elements in
$[0,1[^{\N^*}$ which all tend to $0$ at infinity, uniformly with respect to $\iota'$.
If $\mathbb B = \{(b,b')\in \R \times B\,:\,
b'+[0,1]\subset B,\ b+b'\in B\}$, the sequence of functions
$$
\Big((b,b') \in \mathbb B
\longmapsto \sum\limits_{\nu=0}^N
\binom{N}{\nu} b^\nu (1-b)^{N-\nu} \Psi
\Big( b' + \nu \, \frac{1-\epsilon_{\iota',N}}{N}\Big)\Big)_{N\geq 1}
$$
converges locally uniformly on $\mathbb B$ towards $(b,b') \longmapsto \Psi(b+b')$, the convergence being uniform with respect to the index $\iota'$. Let us call (P) such property. Let $b_0' \in B$ such that $b_0' + [0,1]\subset B$. Proving that any primitive of $\Psi$ in $A$ remains a regular $\C$-supershift amounts to prove that the function
$$
b \in B \longmapsto \Phi(b) = \int_{b_0'}^b \Psi(\beta)\, d\beta
$$
inherits from $\Psi$ the property (P). Theorem \ref{sect2new-thm4} will then follows after the rescaling
$b= \Upsilon^{-1}(a) = (1+a)/2$. We first observe that the sequence of functions
$$
\Big((b,b') \in \mathbb B \longmapsto
\int_{b_0'}^{b'} \Big( \sum\limits_{\nu=0}^N
\binom{N}{\nu} b^\nu (1-b)^{N-\nu} \Psi \Big(\beta + \nu\, \frac{1-\epsilon_{\iota',N}}{N}\Big)\Big)\, d\beta\Big)_{N\geq 1}
$$
converges locally uniformly on $\mathbb B$ towards
$$
(b,b') \longmapsto \int_{b_0}^{b'} \Psi(\beta + b)\, d\xi = \Phi(b+b') - \Phi(b+b'_0).
$$
Therefore, if one proves that the sequence of functions
\begin{equation}\label{sect2new-eq17B}
\Big( b\in \{b\in \R\,:\, b'_0 + b \in B\} \longmapsto \sum\limits_{\nu=0}^N
\binom{N}{\nu} b^\nu (1-b)^{N-\nu}\,
\Phi \Big(b'_0 + \nu\,
\frac{1- \epsilon_{\iota',N}}{N}\Big)\Big)_{N\geq 1}
\end{equation}
converges locally uniformly on $-b'_0 + B = \{b\in \R\,:\, b'_0 + b\in B\}$
towards $b \longmapsto \Phi(b+ b'_0)$, the convergence on each compact subset being
uniform with respect to the index $\iota'$, the fact that $\Phi$ inherits from $\Psi$ the property (P)
will follow. One has on $-b'_0 + B$ that for any $N\geq 2$,
\begin{multline}\label{sect2new-eq18}
\frac{d}{db} \Big(\sum_{\nu=0}^N
\binom{N}{\nu}\, b^\nu (1-b)^{N-\nu} \Phi\Big( b'_0 + \nu\, \frac{1-\epsilon_{\iota',N}}{N}\Big)\Big)
\\
= (1-\epsilon_{\iota',N})\, \sum_{\nu=0}^{N-1} \binom{N-1}{\nu}\, b^\nu (1-b)^{N-1-\nu}\,
\frac{\Phi \Big( b'_0 + (\nu+1)\, \displaystyle{\frac{1 - \epsilon_{\iota',N}}{N}}\Big) -
\Phi \Big( b'_0 + \nu \, \displaystyle{\frac{1 - \epsilon_{\iota',N}}{N}}\Big)}
{\displaystyle{\frac{1-\epsilon_{\iota',N}}{N}}}\\
= (1- \epsilon_{\iota',N})
\int_0^1
\Big(
\sum\limits_{\nu=0}^{N-1}
\binom{N-1}{\nu}
b^\nu (1-b)^{N-\nu}
\Psi \Big( b'_0 + \frac{1-\epsilon_{\iota',N}}{N}\, \xi +
\nu \, \frac{1 - \epsilon_{\iota',N-1}^{[1]}}{N-1}\Big)\Big)\, d\xi,
\end{multline}
where $\epsilon_{\iota',N-1}^{[1]}$ is defined in terms of $\epsilon_{\iota',N}$ as in \eqref{sect2new-eq17A}. Let $K^\circ$ be a compact subset of $-b'_0 + B$,
$\chi_{K^\circ} >0$ such that $b+b'_0 + \tau \in B$ for any $\tau \in [0,\chi_{K^\circ}]$,
and $\varepsilon >0$. It follows from the fact that
$\Psi$ has the property (P) that there exists $N_{K^\circ,\varepsilon}$ that
\begin{multline*}
N \geq N_{K^\circ,\varepsilon}\\ \Longrightarrow
\sup_{\stackrel{(b,\tau)
\in K^\circ \times [0,\chi_{K^\circ}]}{\iota \in I'}}
\Big| \Psi (b'_0 + \tau + b) - \sum\limits_{\nu=0}^{N-1}
\binom{N-1}{\nu} b^\nu (1-b)^{N-\nu}
\Psi \Big( b'_0 + \tau
+ \nu  \frac{1 - \epsilon_{\iota',N-1}^{[1]}}{N-1}\Big)\Big| \leq \varepsilon.
\end{multline*}
One has in particular for $N \geq \max (N_{K^\circ,\varepsilon}, 1/\chi_{K^\circ})$ such that
$$
\Big|
\Psi\Big(b'_0 + \frac{1-\epsilon_{\iota',N}}{N}\, \xi +b\Big)
- \sum\limits_{\nu=0}^{N-1}
\binom{N-1}{\nu}
b^\nu (1-b)^{N-\nu}
\Psi \Big( b'_0 + \frac{1-\epsilon_{\iota',N}}{N}\, \xi +
\nu \, \frac{1 - \epsilon_{\iota',N-1}^{[1]}}{N-1}\Big)\Big|
\leq \varepsilon
$$
for any $(b,\xi) \in K^\circ \times [0,1]$ and any $\iota \in I'$.
It then follows from \eqref{sect2new-eq18} that the sequence of functions
$$
\Big( b\in K^\circ
\longmapsto \frac{d}{db} \Big(\sum_{\nu=0}^N
\binom{N}{\nu}\, b^\nu (1-b)^{N-\nu} \Phi\Big( b'_0 + \nu\, \frac{1-\epsilon_{\iota',N}}{N}\Big)\Big)
\Big)_{N\geq 1}
$$
converges uniformly on $K$ towards $b\longmapsto \Psi(b'_0 +b)$, the convergence being uniform with respect to the index $\iota'\in I'$. This implies that the sequence of corresponding primitives vanishing at
$0$, namely the sequence of functions \eqref{sect2new-eq17B}, converges locally uniformly on
$-b'_0 + B$ towards $\Phi(b+ b'_0)$, which concludes the proof of Theorem \ref{sect2new-thm4}.
\end{proof}


\end{document}